\documentclass[12pt]{article}

\usepackage{color}
\usepackage{amsmath,enumerate}
\usepackage{amssymb}
\usepackage{epsfig}

\usepackage{a4wide}

\author{Louis Esperet\thanks{Laboratoire G-SCOP (Grenoble-INP, CNRS),
    Grenoble, France. This author was partially supported by ANR
    Project HEREDIA, under grant \textsc{anr-10-jcjc-0204-01}.}, Aline
  Parreau\thanks{LIFL (Universit\'e Lille 1, CNRS),
    Villeneuve d'Ascq, France.}}

\title{Acyclic edge-coloring using \\entropy compression}  
\date{\today}

\newenvironment{proof}{\par \noindent \textsc{Proof.} }{\hfill$\Box$\medskip}

\newtheorem{corollary}{Corollary}

\newtheorem{theorem}{Theorem}
\newtheorem{lemma}{Lemma}

\begin{document}

\maketitle

\begin{abstract}
An edge-coloring of a graph $G$ is {\em acyclic} if it is a proper
edge-coloring of $G$ and every cycle contains at least three
colors. We prove that every graph with maximum degree $\Delta$ has an
acyclic edge-coloring with at most $4\Delta-4$ colors, improving the
previous bound of $\lceil 9.62\,(\Delta-1)\rceil$. Our bound results
from the analysis of a very simple randomised procedure using the
so-called \emph{entropy compression method}. We show that the expected
running time of the procedure is $O(mn \Delta^2 \log \Delta)$, where
$n$ and $m$ are the number of vertices and edges of $G$. Such a
randomised procedure running in expected polynomial time was only
known to exist in the case where at least $16\Delta$ colors were
available.

\smallskip

Our aim here is to make a pedagogic
tutorial on how to use these ideas to analyse a broad range of graph
coloring problems. As an application, we also show that every graph with
maximum degree $\Delta$ has a star coloring with $2
\sqrt{2} \,\Delta^{3/2}+\Delta$ colors.
\end{abstract}

\section{Introduction}

An edge-coloring of a graph $G$ is {\em acyclic} if it is a proper
edge-coloring (adjacent edges have different colors) and every cycle
contains at least three colors. The smallest number of colors in an
acyclic edge-coloring of $G$ is the \emph{acyclic chromatic index} of
$G$, denoted by $a'(G)$. A corollary of a general theorem of Alon {\it
  et al.}~\cite{AMR91} from 1991, proved using the Lov\'asz Local
Lemma, is that if $G$ has maximum degree at most $\Delta$, then
$a'(G)\le 64 \Delta$. Molloy and Reed~\cite{MR98} improved the bound
to $16 \Delta$ in 1998, and this was recently improved by Ndreca {\it
  et al.}~\cite{NPS12} to $\lceil 9.62\,(\Delta-1)\rceil$, using a
stronger version of the Local Lemma due to Bissacot {\it et
  al.}~\cite{BFPS11}. Here we improve the bound further to
$4\Delta-4$. Fiam\v{c}ik \cite{F78} (in 1978) and Alon~\emph{et
  al.}~\cite{ASZ01} (in 2001) independently conjectured that the right
bound should be $\Delta+2$ (only one more than the bound of Vizing for
proper edge-coloring).

\smallskip

Let $\gamma>1$ be a fixed real and let
$K=\lceil(2+\gamma)(\Delta-1)\rceil$. We study the following simple
randomised algorithm. Order the edges of $G$ as $e_1,\ldots,e_m$, and
do the following at each step: take the non-colored edge with smallest
index, say $e_i$, and assign it a random color in $\{1,\ldots,K\}$
that does not appear on some edge adjacent to $e_i$ (this will be
slightly modified to allow an easier analysis). If some 2-colored
cycle is created, then uncolor $e_i$ and all the other edges on this
cycle (except two of them, we will understand why later). This way, we
maintain a partial edge-coloring that is acyclic at each step.

Our aim in this paper is to show that this algorithm terminates (every
edge is eventually colored) with positive probability, provided that
$\gamma$ (and thus, $K$) is large enough. This implies that $G$ has an
acyclic edge-coloring with at most $K$ colors.

\smallskip

To analyse the algorithm, we will use ideas that have been developed
to obtain bounds on nonrepetitive coloring of
graphs~\cite{DJKW12,GKM12}. The proofs in these two articles were
inspired by the algorithmic proof of the Local Lemma due
to Moser~\cite{M09} and Moser and Tardos~\cite{MT10}.

\smallskip

We want to insist on the fact that the generality of the work done in
Section~\ref{sec:analysis} makes the technique presented in this paper
(or rather, its precise analysis) easily extendable to a wide variety
of graph coloring problems. We could have made a more general
presentation throughout the whole paper instead of concentrating on
acyclic edge-coloring, but we felt that the paper would completely
loose its pedagogic side. Instead, we chose to present acyclic
edge-coloring first, then give another example (a generalisation of
star coloring, see Section~\ref{sec:star}), and then finally to
explain briefly how these examples could be encompassed in a wider
framework (see Subsection~\ref{sub:ext}). The algorithmic aspects are
analysed in Subsection~\ref{sub:algo}.

\subsection{The algorithm}

In order to analyse such a randomised algorithm running on a
deterministic instance, we will consider it instead as a deterministic
algorithm taking a large vector with random entries as input. Take
some large integer $t$, and consider a vector $F\in \{1,\ldots,\lceil
\gamma (\Delta-1)\rceil\}^t$. At step $i$ of the algorithm, the $i$-th
entry $F_i$ of $F$ will be used to assign a color to the non-colored
edge $e_j$ with smallest index as follows. Let $e_j=uv$, and let
$S=\{1,\ldots,K\}\setminus S'$, where $S'$ is the set of colors
appearing on edges $xy\ne uv$ such that

\begin{enumerate}[(1)]
\item $x=u$ or $x=v$, or
 \item edges $ux$ and $vy$ exist and have the same color.
\end{enumerate}

Observe that that the set $S$ has cardinality at least $\lceil \gamma
(\Delta-1)\rceil$: for any color counted in (2), some color $c$ is
counted at least twice in (1). Moreover since we maintain a proper
coloring at any step (see below), the color $c$ is counted precisely
twice. Hence, $S'$ contains no more colors than the number of edges
adjacent to $e_j$, and so $|S'|\le 2(\Delta-1)$.

We now assign the $F_i$-th smallest element of $S$ to $e_j$.  This
implies that the partial edge-coloring at any step (1) is proper and
(2) has no 2-colored 4-cycle. If a 2-colored cycle (of length at least
6) is created, say $e_{i_1},\ldots,e_{i_{2k}},e_{i_1}$ with
$e_{i_1}=e_j$ and $i_2<i_{2k}$, then uncolor all the edges on this
cycle except $e_{i_2}$ and $e_{i_3}$. Since $e_j$ is uncolored, the
partial edge-coloring remains acyclic.

\smallskip

The key of the analysis of the algorithm is to keep a (compact) record
of each step of the algorithm, in such way that at any step $i$, the
record until step $i$ and the partial coloring at step $i$ are enough
to deduce all the entries $F_j$, $j\le i$. In particular, the set of
all vectors $F$ such that the algorithm did not terminate before step
$t$, is smaller than the set of all possible records of all steps and
partial colorings at step $t$. The total number of choices for $F$ is
$\lceil \gamma (\Delta-1)\rceil^t$ and the number of partial colorings
of $G$ is independent of $t$ (it is at most $(K+1)^m$). Thus, if we
prove that the number of possible records is $o(\lceil \gamma
(\Delta-1)\rceil^t)$ when $t\rightarrow \infty$, this shows that the
algorithm terminates for some input vector. Equivalently, the
randomised version of the algorithm terminates with non-zero
probability.

\medskip

We now precise what we meant by \emph{compact record} of each step of
the algorithm. We define a vector $R$ having $t$ entries as
follows. Assume that at step $i$ of the algorithm, the edge $e_j$ was
colored and a 2-colored cycle (of length at least 6) was created, say
$C=e_{i_1},\ldots,e_{i_{2k}},e_{i_1}$ with $e_{i_1}=e_j$. Observe that
there are at most $(\Delta-1)^{2k-2}$ cycles of length $2k$ containing
$e_j$, so we can fix an order on such cycles (say the lexicographic
order), as $C_1,C_2,\ldots,C_s$, with $s\le (\Delta-1)^{2k-2}$. In this
case we uncolor all edges of $C$ except two, as described above, and
we set the $i$-th entry $R_i$ of $R$ to be equal to the pair
$(k,\ell)$, where $\ell \le s$ is the index of $C$ among all cycles of
length $2k$ containing $e_j$. If no 2-colored cycle is created at step
$i$, $R_i$ is left empty.

\medskip

The algorithm will be analysed in Section~\ref{sec:analysis}. In order
to find good asymptotics for the number of possible records, we will
need to count Dyck words with prescribed descent lengths which is
equivalent to counting rooted plane trees with prescribed number of
children.  The full generality of the counting lemma will be used to
obtain better bounds for the acyclic edge-coloring of graphs without
small cycles and for the star vertex-coloring of graphs
(Section~\ref{sec:star}).

\section{Analysis of the Algorithm}\label{sec:analysis}

We denote by $X_i$ the set of uncolored edges after step $i$, and by
$\Phi_i$ the partial coloring of $G$ after step $i$. Assume that for
some input vector $F$, the algorithm applied on the graph $G$ returns
output $(R,\Phi_t)$.  We now prove that $(R,\Phi_t)$ uniquely
determines $F$.

\begin{lemma}\label{lem:uncolored}
At each step $i$, the set $X_i$ is uniquely determined by the record
$(R_j)_{j \le i}$.
\end{lemma}

\begin{proof}
We prove the result by induction on $i$. First observe that the set
$X_1$ is the set of all edges except $e_1$. Assume now that $i\ge
2$. By the induction, $X_{i-1}$ is uniquely determined, so in
particular the uncolored edge with smallest index before step $i$, say
$e_j$, is uniquely determined. If $R_i$ is empty,
$X_i=X_{i-1}\setminus e_j$. If $R_i$ is not empty, say $R_i=(k,\ell)$,
we know which cycle of length $2k$ containing $e_j$ was a 2-colored
cycle, and which edges from this cycle were uncolored. So $X_i$ is
uniquely determined also in this case.
\end{proof}

\begin{lemma}\label{lem:input}
At each step $i$, the application that assigns to each input $(F_j)_{j\le
  i}$ the output $((R_j)_{j\le i},\Phi_i)$ is injective.
\end{lemma}

\begin{proof}
We prove by induction on $i$ that the record $(R_j)_{j\le i}$ and the
partial coloring $\Phi_i$ uniquely determine the input $(F_j)_{j\le
  i}$ that produced such record and coloring. After the first step, the
color of the only colored edge in $\Phi_1$ is equal to $F_1$. Next,
assume that $i \ge 2$. By Lemma~\ref{lem:uncolored} we know $X_i$ and
$X_{i-1}$. In particular, we know the edge $e_j$ that is colored at
step $i$. 

Assume first that $R_i$ is empty. Then $\Phi_{i-1}$ is obtained from
$\Phi_i$ by simply uncoloring $e_j$. By the induction, it follows that
$(F_j)_{j\le i-1}$ is uniquely determined, and all that remains is to
find $F_i$. Let $c\in \{1,\ldots,K\}$ be the color of $e_j=uv$ in
$\Phi_i$, and let $a$ be the number of different colors $\{i \, |\, i
< c\}$ appearing in the coloring $\Phi_{i-1}$ on (1) edges adjacent to
$e_j$ or (2) edges $xy$ such that $ux$ and $yv$ are edges of $G$ and
have the same color. Then $F_i=c-a$.

Now assume that $R_i=(k,\ell)$, with $\ell$ corresponding to some
cycle of length $2k\ge 6$, say $C=e_{i_1},\ldots,e_{i_{2k}},e_{i_1}$ with
$e_{i_1}=e_j$ and $i_2<i_{2k}$. Then since $C$ is 2-colored when $e_j$
is assigned its color, the coloring $\Phi_{i-1}$ is obtained from
$\Phi_i$ by coloring $e_{i_{5}},e_{i_{7}},\ldots,e_{i_{2k-1}}$ with
color $\Phi_i(e_{i_3})$ and $e_{i_{4}},e_{i_{6}},\ldots,e_{i_{2k}}$
with color $\Phi_i(e_{i_2})$. Moreover, $e_j$ received color
$\Phi_i(e_{i_3})$ at step $i$ just before being uncolored. As above,
we conclude using the induction that since $\Phi_{i-1}$ is uniquely
determined, so is $(F_j)_{j\le i-1}$, and we obtain $F_i$ from the color
assigned to $e_j$ at this step as in the previous paragraph.
\end{proof}

Let $\mathcal F_t$ be the set of vectors $F$ such that at step $t$ of
the algorithm, the graph $G$ has not been completely colored (in other
words, $X_t$ is not empty). By definition of $F$, $|\mathcal F_t|\leq
\lceil \gamma (\Delta-1)\rceil^t$ and if the inequality is strict,
then $G$ has an acyclic edge-coloring with
$K=\lceil(2+\gamma)(\Delta-1)\rceil$ colors.

Let $\mathcal R_t$ be the set of records $R$ that can be produced with
inputs from $\mathcal F_t$. Since there are at most $(K+1)^m$ partial
colorings $\Phi_t$ of $G$, the two previous lemmas have the following
direct consequence:

\begin{lemma}\label{lem:cor}
$|\mathcal F_t|\le (K+1)^m |\mathcal R_t|$.
\end{lemma}

We will now compute $|\mathcal R_t|$ and show that for $t$ large
enough, $|\mathcal F_t|$ is smaller than the set of all possible
vectors, meaning that there is a vector $F$ for which the algorithm
terminates.

\medskip

Recall that a 2-colored cycle that is partially uncolored at some step
is recorded by a pair $(k,\ell)$, where the cycle has length $2k\ge 6$, and
index $\ell$ among the at most $(\Delta-1)^{2k-2}$
cycles of length $2k$ containing the current edge. Hence $\ell \le
(\Delta-1)^{2k-2}$. 

Consider a word $w=w_1\ldots w_{2k-2}$ of length $2k-2$ on the
alphabet $\mathcal A=\{1,\ldots,\Delta-1\}$, and define
$\theta_k(w)=1+\sum_{i=1}^{2k-2} (w_i-1) \,(\Delta-1)^{i-1}$. Then the
function $\theta_k$ has range in $1,\ldots,(\Delta-1)^{2k-2}$ and is
bijective.

Let $R\in \mathcal R_t$. Define $R^*=(R_i^*)_{i\le t}$ as the
following sequence of $t$ words on the alphabet $\mathcal A^*=\mathcal
A\cup \{0\}$: for any $1 \le i\le t$, if $R_i$ is empty, then
$R^*_i=0$. Otherwise $R_i=(k,\ell)$ for some $k,\ell$ and we set
$R^*_i$ to be the concatenation of $0$ and ${\theta^{-1}}_k(\ell)$. We
now consider the sequence of words $R^*$ as a word $R^\bullet$
(concatenating all the entries in order), and define $R^\circ$ as the
word on $\{0,1\}$ obtained from $R^\bullet$ by the morphism
$\kappa(x)=0$ if $x= 0$ and $\kappa(x)=1$ otherwise. For instance, if
$\Delta=4$ and \smallskip

\begin{tabular}{l c l} \smallskip
$R$ & $=$ & $(\emptyset,\emptyset,\emptyset,\emptyset,
\emptyset,(3,4),\emptyset, \emptyset,\emptyset,(3,15))$, then we have\\ 
    $R^*$ & $=$ & $(0,0,0,0,0,01211,0,0,0,03221)$,\\
    $R^\bullet$ & $=$ & $000000121100003221$, and\\
$R^\circ$ & $=$ & $000000111100001111.$\\
\end{tabular}
\smallskip

Observe that
the function $R^* \mapsto R^\bullet$ is an injection since every entry
of $R^*$ starts with a 0 and there are no other 0's in words of
$R^*$. It follows that the function $R \mapsto R^\bullet$ is also an
injection. We now make a couple of observations on words $R^\circ$,
for $R\in \mathcal R_t$.

\medskip

A \emph{partial Dyck word} is a word $w$ on the alphabet $\{0,1\}$
such that any prefix of $w$ contains at least as many $0$'s as
$1$'s. A \emph{Dyck word} of length $2t$ is a partial Dyck word with
$t$ 0's and $t$ 1's. A \emph{descent} in a (partial) Dyck word is a
maximal sequence of consecutive $1$'s.

\begin{lemma}\label{lem:partialdyck}
For any $R\in\mathcal R_t$, the word $R^\circ$ is a partial Dyck word
with $t$ 0's and $t-r$ 1's, where $r$ is the number of colored edges
after step $t$. Moreover, all descents in $R^\circ$ are even, and if
every cycle of $G$ has length at least $2\ell+1$, for some $\ell\ge
1$, then all descents in $R^\circ$ have length at least
$\max(4,2\ell)$.
\end{lemma}

\begin{proof}
When we read $R^\circ$ from left to right, every 0 in $R^\circ$
corresponds to an edge that is assigned a color according to the
vector $F$, while every 1 corresponds to the process of uncoloring
some colored edge. Indeed, every word $\kappa(R^*_i)$ is of the form
$01^{2k-2}$, for some $k\ge 1$; if $k=1$ it means that at step $i$ of
the algorithm, the edge that is colored is not involved in any
conflict, while if $k \ge 2$ the colored edge is contained in some
2-colored cycle of length $2k$. In this case we uncolor $2k-2$ edges,
which is precisely the number of 1's in $\kappa(R^*_i)$. Since we
cannot uncolor more edges than the number of colored edges, the result
of the first part follows.  The second part follows from the fact that
if all cycles have length at least $2\ell +1$, all even cycles have
length in at least $2\ell+2$.  A 2-colored cycle in the algorithm has
length at least 6, so each descent is even and has length at least
$\max(4,2\ell)$.
\end{proof}

Let $R\in \mathcal R_t$. If the word $R^\circ$ has $t-r$ 1's, then the
preimage of $R^\circ$ under the function $R \mapsto R^\circ$ has
cardinality at most $(\Delta-1)^{t-r}$. This follows from the fact that $R
\mapsto R^*$ and $R^* \mapsto R^\bullet$ are injections, and each 1 in
$R^\circ$ corresponds to an element of $\{1,\ldots,\Delta-1\}$ in
$R^\bullet$.

Let $\mathcal R^\circ_t= \{R^\circ\,|\,R\in\mathcal R_t\}$. The
preceding remark, together with Lemma~\ref{lem:partialdyck} (more
precisely, the fact that the number of 1's is at most the number of
0's in $R^\circ$) show that $|\mathcal R_t|\le (\Delta-1)^t |\mathcal
R^\circ_t|$. Hence, Lemma~\ref{lem:cor} implies the following.

\begin{lemma}\label{lem:cor2}
$|\mathcal F_t|\le (K+1)^m (\Delta-1)^t |\mathcal R^\circ_t|$.
\end{lemma}

Our objective is now to count partial Dyck words having the properties
described in Lemma~\ref{lem:partialdyck}. To make the computation
easier, we will in fact count \emph{Dyck words} with these
properties. The next lemma shows that counting these two objects is
almost equivalent, provided that $r$ (the difference between the
number of 0's and 1's in the partial Dyck word) is not too large.

\begin{lemma}\label{lem:countingpartial}
Let $t$ and $r\leq t$ be integers, and let $E\ne\{1\}$ be a non-empty
set of non-negative integers. Let $C_{t,r,E}$ (resp. $C_{t,E}$) be the
number of partial Dyck words with $t$ 0's, $t-r$ 1's (resp. Dyck words
with length $2t$), and all descents having length in $E$. Then
$C_{t,r,E} \le C_{t+r(s-1),E}$, where $s=\min (E\setminus \{1\})$.
\end{lemma}

\begin{proof}
Let ${\mathcal D}_{t,r,E}$ (resp. ${\mathcal D}_{t,E}$) be the set of
partial Dyck words with $t$ 0's, $t-r$ 1's (resp. Dyck words with
length $2t$), and all descents having length in $E$. Let $\psi:
{\mathcal D}_{t,r,E} \to {\mathcal D}_{t+r(s-1),E}$ be the application that
    appends to the end of a word the word $(0^{s-1}1^s)^r$.  Observe
    that the application $\psi$ is well-defined and injective. The
    result follows.
\end{proof}

There are various ways to compute asymptotics for $C_{t,E}$, in
particular by finding bijections with well-known structures. We use
here a bijection with rooted plane trees\footnote{A \emph{rooted plane
    tree} is a tree embedded in the plane in which a given vertex (the
  \emph{root}) is specified. The embedding gives an order on the
  subtrees attached to each node.}.

\begin{lemma}
The number $C_{t,E}$ of Dyck words with length $2t$ and all descents
in $E$ is equal to the number of rooted plane trees on $t + 1$
vertices such that the degree (number of children) of each vertex is
in $E\cup\{0\}$.
\end{lemma}

\begin{proof}
There are bijections between the following three objects for any
integer $t$, proving the lemma:

\begin{enumerate}
\item rooted plane trees on $t+1$ vertices such that the degree of
  each vertex is in $E \cup \{0\}$;

\item Dyck words
of length $2t$ in which the length of any maximal sequence of consecutive
$0$'s is in $E$;

\item Dyck words
of length $2t$ such that the length of each descent is in $E$.
\end{enumerate}

The bijection between items 1 and 2 is as follows: in a DFS walk of
the tree, encode each vertex (except the very last one) having $i$
children by $0^i1$. The word obtained is a Dyck word in which every
maximal sequence of consecutive $0$'s is in $E$, and this application
is clearly a bijection. The bijection between items 2 and 3 proceeds
by taking the mirror of the word and interchanging 0's and 1's.
\end{proof}

We now use generating functions and the method described by Drmota
in \cite{D} (see also the book of Flajolet and Sedgewick~\cite{FS}) to
estimate $C_{t,E}$. Let $X_E(z)$ be the ordinary generating function
(OGF) associated to the number of rooted plane trees on $t+1$ vertices
such that the degree (number of children) of each vertex is in $E \cup
\{0\}$.  By the previous lemma, $X_E(z)=z\sum_{t\in \mathbb{N}}
C_{t,E}\, z^t$.  But a rooted plane tree as above is either a unique
vertex (the root), or the root together with a sequence of $i$ rooted
plane trees (such that the degree of each vertex is in $E \cup
\{0\}$), for some $i\in E$. It follows that $X_E$ satisfies the
equation $X_E(z)=z(1+\sum_{i\in E}X_E(z)^i)=z\,\phi_E(X_E(z))$, with
$\phi_E(x)=1+\sum_{i\in E}x^i$.

\smallskip

The next lemma is a direct corollary of~\cite[Theorem 5]{D} (see
also~\cite[Theorem VI.6]{FS} where the result is stated in the
specific case of aperiodic schemes). We just need to observe that for
any non-empty set $E\ne\{1\}$ of nonnegative integers, all the
coefficients of $\phi_E$ are nonnegative and $\phi_E(x)$ is not linear
in $x$.

\begin{lemma}\label{lem:dyck}
Let $E\ne\{1\}$ be a non-empty set of nonnegative integers such that the
equation $\phi_E(x)-x \phi_E'(x)=0$ has a solution $x=\tau$ with $0 <
\tau <R$, where $R$ is the radius of convergence of $\phi_E$. Then
$\tau$ is the unique solution of the equation in the open interval
$(0,R)$. Moreover there is a constant $c_E$ such that $C_{t,E} \le c_E
\, \gamma^t \,t^{-3/2}$, where $\gamma=\phi_E'(\tau)=\phi_E(\tau)/\tau$.
\end{lemma}

We can now derive bounds on the acyclic chromatic index of
graphs. Recall that the girth of a graph $G$ is the length of a
shortest cycle of $G$ (if $G$ is acyclic, its girth is $+\infty$). 

\begin{theorem}\label{th:main}
Let $\ell\ge 1$ be a fixed integer, and let $k=\max(2,\ell)$. Then the
polynomial $P(x)=(2k-3)x^{2k+2}+(1-2k)x^{2k}+x^4-2x^2+1$ has a unique
root $\tau$ in the open interval $(0,1)$. Moreover, every graph with
maximum degree $\Delta$ and girth at least $2\ell+1$ has an acyclic
edge-coloring with at most $\lceil (2+\gamma) (\Delta-1)\rceil$ colors,
where $\gamma=(\tau^{2k}-\tau^{2}+1)/(\tau-\tau^3)$.
\end{theorem}

\begin{proof}
Let $E=2 \mathbb{N}+2k$. Then $\phi_E(x)=1+\sum_{i \in
  E}x^i=1+\tfrac{x^{2k}}{1-x^2}$. It follows that
$\phi'_E(x)=(2kx^{2k-1}-(2k-2)x^{2k+1})/(1-x^2)^2$, and the
characteristic equation $\phi_E(x)-x\phi'_E(x)=0$ is equivalent to
$P(x)=0$. The radius of convergence of $\phi_E$ is 1 and since
$P(0)=1$ and $P(1)=-2$ the polynomial $P$ has a root $\tau$ in the
open interval $(0,1)$. By Lemma~\ref{lem:dyck}, this is the unique
root in $(0,1)$. Lemma~\ref{lem:dyck} also implies that for some
constant $c_E$, $C_{t,E}\le c_E\, \gamma^t \,t^{-3/2}$, where
$\gamma=\phi'_E(\tau)=\phi_E(\tau)/\tau=
(\tau^{2k}-\tau^{2}+1)/(\tau-\tau^3)$. 

In order to prove the theorem, we just need to show the existence of a
vector $F \in \{1,\ldots,\lceil\gamma (\Delta-1)\rceil\}^t$ such that
the algorithm taking $G$ and $F$ as inputs yields an acyclic
edge-coloring of $G$. In other words, all the edges are colored at
step $t$. As before, let $\mathcal F_t$ be the set of vectors $F$ for
which some edges remain uncolored at step $t$, and let $m$ be the
number of edges of $G$.  By Lemma~\ref{lem:cor2}, $|\mathcal F_t|\le
(\lceil(2+\gamma)(\Delta-1)\rceil+1)^m (\Delta-1)^t |\mathcal
R^\circ_t|$. Observe that for any $R\in \mathcal R_t$, the number of
0's and 1's in each prefix of $R^\circ$ differ by at most $m-1$, since
at most $m-1$ edges are colored at each step of the algorithm. By
Lemmas~\ref{lem:partialdyck} and~\ref{lem:countingpartial}, it implies
that $ |\mathcal R^\circ_t| \leq \sum_{r=0}^{m-1} C_{t+r(2k-1),E} \leq
c'_E\, \gamma^{t+m(2k-1)} \,t^{-3/2}$, where
$c'_E=c_E/(\gamma^{2k-1}-1)$. It follows that $|\mathcal F_t|\le c'_E
(\lceil(2+\gamma)(\Delta-1)\rceil+1)^m (\Delta-1)^t\gamma^{t+m(2k-1)}
t^{-3/2}$, and $|\mathcal F_t|/\lceil\gamma
(\Delta-1)\rceil^t$ tends to 0 as $t$ goes to infinity. In particular,
for $t$ large enough $|\mathcal F_t| < \lceil\gamma
(\Delta-1)\rceil^t$, which means that for some vector $F$ the
algorithm terminates in less than $t$ steps and yields an acyclic
edge-coloring of $G$ with at most $\lceil(2+\gamma) (\Delta-1)\rceil$
colors.
\end{proof}

Muthu {\it et al.}~\cite{MNS07} proved in 2007 that graphs of maximum
degree at most $\Delta$ and girth at least 9 have an acyclic
edge-coloring with at most $6\Delta$ colors, and for graphs with girth
at least 220 the bound was improved to $4.52 \,\Delta$. Ndreca {\it et
  al.}~\cite{NPS12} recently showed the following bounds for the
acyclic edge-coloring of graphs $G$ with maximum degree $\Delta$ and
girth at least $g$: $a'(G)\le \lceil 9.62\,(\Delta-1)\rceil$,
$a'(G)\le \lceil 6.42\,(\Delta-1)\rceil $ if $g\ge 5$, $a'(G)\le
\lceil 5.77\,(\Delta-1)\rceil$ if $g\ge 7$, and $a'(G)\le \lceil 4.52
\,(\Delta-1)\rceil$ if $g\ge 53$. The following direct corollary of
Theorem~\ref{th:main} significantly improves all these bounds.

\begin{corollary}\label{cor:main}
Let $G$ be a graph with maximum degree $\Delta$ and girth $g$. Then
\begin{enumerate}
\item $a'(G)\le 4\Delta-4$;
\item if $g\ge 7$, $a'(G)\le \lceil 3.74 \,(\Delta-1)\rceil$;
\item if $g\ge 53$, $a'(G)\le \lceil 3.14 \,(\Delta-1)\rceil$;
\item if $g\ge 220$, $a'(G)\le \lceil 3.05 \,(\Delta-1)\rceil$.
\end{enumerate}
\end{corollary}

The constants appearing in the computations leading to
Corollary~\ref{cor:main} are given in Table~\ref{tab1}.

\begin{table}
\centering
\begin{tabular}{|c|c|c|c|c|}
  \hline $g$ & $E$ & $P(x)$ & $\tau$ & $\gamma$ \\ \hline\hline 3 & $2
  \mathbb{N}+4$ & $x^6-2x^4-2x^2+1$ & $\tfrac12(\sqrt{5}-1)$ & 2
  \\\hline 7 & $2 \mathbb{N}+6$ & $3x^8-5x^6+x^4-2x^2+1$ &0.66336 &
  1.73688 \\\hline 53 & $2 \mathbb{N}+52$ &
  $49x^{54}-51x^{52}+x^4-2x^2+1$ &0.89610 & 1.13481 \\\hline 220 & $2
  \mathbb{N}+218$ & $215x^{220}-217x^{218}+x^4-2x^2+1$ & 0.96341 &
  1.04225 \\ \hline
\end{tabular}\caption{Computations in
  Corollary~\ref{cor:main}}\label{tab1} 
\end{table}

\section{Star coloring}\label{sec:star}

We now apply the analysis of the algorithm to star coloring of
graphs. A \emph{star coloring} of a graph $G$ is a proper coloring of
its vertices such that any two color classes induce a forest of
stars. Equivalently, every path on four vertices contains at least three
colors. The \emph{star chromatic number} of a graph $G$, denoted by
$\chi_s(G)$, is the minimum number of colors in a star coloring of
$G$. Fertin {\it et al.}~\cite{FRR04} proved that for every graph $G$
with maximum degree $\Delta$, $\chi_s(G)\le 20 \Delta^{3/2}$, and
that this bound is best possible up to a polylogarithmic factor: for
some absolute constant $C$, there are graphs with maximum degree
$\Delta$ requiring $C\, \Delta^{3/2}/(\log \Delta)^{1/2}$ colors in
any star coloring. Recently, Ndreca {\it et al.}~\cite{NPS12} showed
that for every graph $G$ with maximum degree $\Delta$, $\chi_s(G)\le
4.34 \,\Delta^{3/2}+1.5 \,\Delta$.

We will show how to divide this bound by $\tfrac32$
using a variant of the algorithm analysed in this paper. Instead of
considering star coloring, we will consider the following more general
concept: a \emph{star-$k$ coloring} of a graph $G$ is a proper
vertex-coloring of $G$ such that every path on $2k$ vertices contains at least three colors. A star coloring is the same as a star-2 coloring.

\begin{theorem}
For every $k\ge 2$, every graph $G$ with maximum degree $\Delta$ has a
star-$k$ coloring with at most
$C_{2k-2}\,k^{\frac1{2k-2}}\,\Delta^{\frac{2k-1}{2k-2}}+\Delta$
colors, where $C_\ell=\ell\,(\ell-1)^{\frac{1}{\ell}-1}$.
\end{theorem}

\begin{proof}
Let $\ell=2k-2$ and $K=C_{\ell}\,k^{\frac1{\ell}}\,\Delta^{1+\frac1{\ell}}$. We
order the vertices as $v_1,\ldots,v_n$, and at each step we consider
the non-colored vertex with smallest index, say $v_j$, pick a random
integer $r$ in $1,\ldots,K$, and assign $v_j$ the $r$-th color in the
set $\{1,\ldots,K+\Delta\}$ that does not appear in the neighborhood of
$v_j$. If some path of length $2k$ is now 2-colored, we choose such a
path and uncolor $v_j$ and all the other vertices on the path, except
two consecutive ones. Hence, the coloring remains a star-$k$
coloring at each step.  The analysis is the same as above. The two
vertices on the 2-colored path that are not uncolored are enough to
recover the colors of all the other vertices on the path, including
$v_j$. It follows that the complete record until step $i$ together
with the partial coloring at step $i$ are enough to deduce all the
random choices until step $i$. Hence, we only need to show that their
are $o(K^t)$ possible complete records at step $t$.

Every vertex is contained in at most $k \Delta^{2k-1}$ paths on $2k$
vertices, so the 2-colored path containing $v_j$ that will be
partially uncolored at this step can be recorded using a word of
length $\ell=2k-2$ on the alphabet
$1,\ldots,k^{\frac1{\ell}}\,\Delta^{1+\frac1{\ell}}$ ($\ell$ is
precisely the number of vertices that are uncolored at this
step). Applying the same morphism as in Section~\ref{sec:analysis} we
obtain a partial Dyck word in which every descent has length precisely
$\ell$.

It can be proven fairly easily with a bijective argument that the
number of Dyck words of length $2t$ in which every descent has length
precisely $\ell$ is $\frac{1}{t+1}{t+1 \choose t/\ell}$ and,
using Stirling formula, its asymptotic value is $ct^{-3/2}C_{\ell}^t$
where $c$ is a constant. We omit the details, and rather present how
these asymptotics can be directly deduced from the framework of
Section~\ref{sec:analysis}. In this framework, we have $E=\{\ell\}$ and we
want the asymptotic behavior of $C_{t,E}$. We have
$\phi_E(x)=1+x^\ell$, whose radius of convergence is $+\infty$. Then
$\tau=(\ell-1)^{-1/\ell}$ is the only solution of the characteristic
equation $\phi_E(x)-x\phi'_E(x)=0$ in the interval
$(0,+\infty)$. Since $\phi'_E(\tau)=\ell (\ell-1)^{1/\ell-1}$,
Lemma~\ref{lem:dyck} implies that for some some constant $c$,
$C_{t,E}\le c\, C_\ell^t \,t^{-3/2}$.

We can now conclude that for some constant $c'$ depending only on $c$
and the number of vertices of $G$, the number of possible records of
the algorithm after $t$ steps is at most $c'\, C_\ell^t \,t^{-3/2} \,
(k^{\frac1{\ell}}\,\Delta^{1+\frac1{\ell}})^t= c'\,t^{-3/2}\, K^t$. It
follows that $G$ has a star-$k$ coloring with $K+\Delta$ colors.
\end{proof}

This theorem has the following immediate corollary, improving on~\cite{NPS12}.

\begin{corollary}\label{cor:star}
For every graph $G$ with maximum degree $\Delta$, $\chi_s(G) \le 2
\sqrt{2} \,\Delta^{3/2}+\Delta$.
\end{corollary}

\section{Conclusion}

\subsection{Extensions}\label{sub:ext}

The method presented in this paper can be applied to any
vertex-coloring (or edge-coloring) that can be defined as a coloring
where some configurations of colors are forbidden. By a
\emph{configuration}, we mean a graph $H_i$ with a specific vertex-
(or edge-) coloring $c_i$, and we seek a coloring $c$ of a graph $G$,
such that for any $i$, and any copy $H$ of $H_i$ in $G$, the
restriction of the coloring $c$ to $H$ is not congruent to $c_i$ (two
colorings of the same graph are \emph{congruent} if one can be
obtained from the other one by a permutation of the color names). For
instance, in the case of star coloring, there would only be two
configurations: $H_1$ (a single edge with both ends having the same
color) and $H_2$ (a properly 2-colored path on 4 vertices).

Assume that for any vertex $v$ of $H_i$, there are $k_i$ fixed
vertices different from $v$ in $H_i$ for which, if we know their
color, there is a unique way to extend this partial coloring to a
coloring of $H_i$ congruent to $c_i$.  For any $i$, let
$\ell_i=|V(H_i)|-k_i$, and let $E=\{\ell \in \mathbb N \, |\, \exists
i, \ell_i=\ell\}$. For $\ell \in E$, let $d_\ell$ be the maximum over
all vertices $v$ of $G$, of the number of subgraphs containing $v$ and
isomorphic to some $H_i$ with $\ell_i=\ell$. Let $\gamma$ be defined
as in Lemma~\ref{lem:dyck} using this set $E$. Using the same analysis
as before, we can prove that there is a coloring of the graph with
$\gamma\cdot \sup_{\ell \in E} d_\ell^{1/\ell}$ colors, so that no
copy of $H_i$ as a coloring congruent to $c_i$, for any $i$.

\paragraph{Example 1: Star coloring} Taking $H_1$ and $H_2$ as
defined above, we obtain $k_1=1$ and $k_2=2$, and thus $\ell_1=1$ and
$\ell_2=2$. It follows that $E=\{1,2\}$, and so $\gamma=3$, and if $G$
has maximum degree $\Delta$ we have $d_1\le\Delta$ and
$d_2\le2\Delta^3$. It follows that $\chi_s(G)\le
3\sqrt{2}\,\Delta^{3/2}$. This is not as good as the bound of
Corollary~\ref{cor:star}, though. The reason is that in the previous
section we did not consider $H_1$ and used a different (and less
expensive) tool to keep the coloring proper at any step.

\paragraph{Example 2: Nonrepetitive coloring} Here all paths on an
even number of vertices where the sequence of colors of the first half
of the path is repeated on the second half are forbidden. If we only
consider paths on 2 and 4 vertices, this corresponds exactly to star
coloring. The forbidden configurations are paths $H_i$ of length $2i$,
$i\ge 1$, with colorings $c_i$ such that for any two vertices $x$ and
$y$ at distance $i$ in $H_i$, $c_i(x)=c_i(y)$. We obtain that for each
$i\ge 1$, $k_i=\ell_i=i$ and $E=\mathbb N +1$. It implies that
$\phi_E(x)=1+\tfrac{x}{1-x}$, which yields a constant $\gamma=4$ in the
computation of Lemma~\ref{lem:dyck}. For any $i\ge 1$, we have $d_i\le
i\, \Delta^{2i-1}$, therefore every graph of maximum degree $\Delta$ has a
nonrepetitive coloring with $4\cdot \sup_{\ell\ge 1} \{\ell^{1/\ell}
\Delta^{2-1/\ell}\}\leq (4+o(1))\Delta^2$. In~\cite{DJKW12}, the
authors analysed this randomised procedure more precisely and obtained
a bound of $(1+o(1))\Delta^2$.

\paragraph{Example 3: Acyclic edge-coloring} In this last example, we
compare the bound obtained by a direct application of the framework
above with the bound proved in Theorem~\ref{th:main}. The forbidden
configurations are $H_1$ (a path on two edges having the same color),
and for any $i\ge 2$, a properly 2-colored cycle $H_i$ on $2i$
edges. We obtain that $k_1=\ell_1=1$ and for each $i\ge 2$, $k_i=2$
and $\ell_i=2i-2$. It implies that $E=\{1\} \cup 2\mathbb N +2$ and
$\phi_E(x)=1+x+\tfrac{x^2}{1-x^2}$, which yields a constant
$\gamma=3.6$ in the computation of Lemma~\ref{lem:dyck}. We have
$d_1\le 2\Delta$ and for any $i\ge 2$, $d_{2i-2}\le \Delta^{2i-2}$,
therefore every graph of maximum degree $\Delta$ has an acyclic
edge-coloring with $3.6\cdot 2 \Delta=7.2\,\Delta$ colors. This is of
course not as good as the bound of Theorem~\ref{th:main}, in which
small configurations are taken care of in a different way to minimize
their influence on the final bound.

\medskip

The algorithm and the different bounds in the applications have been
formulated in terms of coloring for the sake of clarity but it is not
difficult to see that everything works in the more general context of
list coloring. Hence, all the bounds obtained here also hold for
acyclic/star choosability.

\subsection{Algorithmic remarks}\label{sub:algo}

By Corollary~\ref{cor:main}.1, the acyclic chromatic index of every
graph $G$ with maximum degree $\Delta$ is at most $4 \Delta-4$. To
prove this result, we showed that for $t$ large enough, our random
procedure colors $G$ in at most $t$ steps with non-zero probability
(more precisely, with probability tending to 1 as $t\rightarrow
\infty$). In the proof of Theorem~\ref{th:main} the value $t$ for
which this probability is non-zero is exponential in the number of
edges, but if we allow one more color (i.e. we take $K=4\Delta-3$
instead of $4\Delta-4$) we obtain that the probability that the
algorithm stops in at most $t\geq t_0$ steps (with $t_0=\frac{m\log(32
  \Delta)}{\log(1+1/2\Delta)}$) is at least $$1-\frac{(4\Delta-2)^{m}
  (\Delta-1)^t \,2^{t+3m}}{(2\Delta-1)^t}\geq
1-\frac{(32\Delta)^{m}}{(1+\tfrac1{2\Delta})^t}=1-e^{-\lambda
  (t-t_0)}$$ where $\lambda = \log(1+\tfrac1{2\Delta})$. This
corresponds to an exponential distribution, therefore the expected
number of steps is at most $t_0+\frac{1}{\lambda} =\frac{m\log(32
  \Delta)+1}{\log(1+1/2\Delta)}=O(m \Delta \log \Delta)$.

The previous remark also holds in the full generality of
Subsection~\ref{sub:ext}: if one allows one more color than the number
of colors guaranteed by the general technique, then the expected
number of steps becomes polynomial in the size of the graph. The issue
is that in general, there is no clear way to perform each step of the
algorithm in polynomial time. An example is the case of nonrepetitive
coloring considered in the previous subsection. It was proved
in~\cite{MS09} that deciding whether a given coloring of a graph is
nonrepetitive is Co-NP-complete, so there is no polynomial time
algorithm finding a repetitive path in a colored graph unless P=NP.

However in the case of acyclic edge-coloring each step can be
performed in time $O(n \Delta)$ (the time it takes to find a 2-colored
cycle containing a given edge, if such a cycle exists, in a graph with
a proper edge-coloring). We only need to modify slightly the way we
encode cycles in the record $R$ (for each vertex $u$, we label the
ordered pairs of neighbors $(u,v)$ by $1,\ldots,\Delta$, and a cycle
$u_1u_2\ldots u_{2k}u_1$ of length $2k$ containing the edge $u_1u_2$
is uniquely determined by the sequence of $2k-2$ labels of consecutive
ordered pairs $(u_2,u_3),(u_3,u_4),\ldots,(u_{2k-1}u_{2k})$). 

It follows that the overall expected running time is $O(mn \Delta^2
\log \Delta)$. In particular, if $\Delta$ is fixed, the expected
running time of our procedure is $O(n^2)$.

Note that this improves procedures producing an acyclic edge-coloring
in expected polynomial time given by Molloy and Reed in 1998 using $20
\Delta$ colors~\cite{MR98}, and recently by Haeupler, Saha, and
Srinivasan using $16\Delta$ colors~\cite{HSS11}. The latter result was
proved using a refined analysis of the constructive proof of Moser and
Tardos~\cite{MT10}.

\paragraph{Acknowledgement}

We would like to thank Mireille Bousquet-M\'elou for mentioning the
bijection between Dyck words and plane rooted trees, Jakub
Kozik for his suggestion improving our original bound of $(2+\tfrac32
\sqrt{3})\, \Delta\approx 4.6 \, \Delta$ to $4 \Delta$, and Zhentao Li
and Jan Volec for discussions related to algorithmic issues.


\end{document}